\documentclass[12pt]{amsart}

\usepackage{array,amssymb,amsxtra,latexsym,graphics,psfrag,epsfig,ifthen}
\usepackage[all]{xy}
\SelectTips {cm}{}
\usepackage{bbm,bm}
\usepackage{rotating}
\usepackage[colorlinks,final,backref=page,hyperindex]{hyperref}

\usepackage{xcolor}

\usepackage[newitem,newenum,neverdecrease]{paralist}
\makeatletter
\if@plflushright
  
\else
  
\fi
\renewcommand{\@asparaenum@}{%
  \expandafter\list\csname label\@enumctr\endcsname{%
    \usecounter{\@enumctr}%
    \labelwidth\z@
    \labelsep.5em
    \leftmargin\z@
    \parsep\parskip
    \itemsep\z@
    \topsep\z@
    \partopsep\parskip
    \itemindent\parindent
    \advance\itemindent\labelsep
    \def\makelabel##1{\upshape ##1}}}
\makeatother
\usepackage{psfrag}
\usepackage{longtable}
\theoremstyle{plain}
\newtheorem{theorem}{Theorem}[section]
\newtheorem{proposition}[theorem]{Proposition}
\newtheorem{lemma}[theorem]{Lemma}
\newtheorem{corollary}[theorem]{Corollary}

\newtheorem*{theorem*}{Theorem}

\theoremstyle{definition}

\newtheorem{example}[theorem]{Example}

\theoremstyle{remark}

\newtheorem*{remark}{Remark}

\numberwithin{equation}{section}

\newcommand{\Kb}{\Bbbk}

\newcommand{\Nb}{\mathbb{N}}
\newcommand{\Rb}{\mathbb{R}}
\newcommand{\Fb}{\mathbb{F}}

\newcommand{\bdot}{\bm\cdot}

\DeclareMathOperator{\area}{sch}

\newcommand{\id}{\mathrm{id}}
\newcommand{\onto}{\twoheadrightarrow}
\newcommand{\into}{\hookrightarrow}
\newcommand{\iso}{\cong}

\newcommand{\xyinc}{\ar@{^{(}->}}
\newcommand{\xyrinc}{\ar@{_{(}->}}
\newcommand{\xyonto}{\ar@{->>}}
\newcommand{\xytwo}{\ar@{<->}}

\newcommand{\map}[1]{\xrightarrow{#1}}

\newcommand{\abs}[1]{\lvert#1\rvert} 

\DeclareMathOperator{\rank}{rank}

\def\llb{{[\![}}
\def\rrb{{]\!]}}

\newcommand{\Fset}{\mathsf{set^{\times}}}
\newcommand{\Vect}{\mathsf{Vec}} 
\newcommand{\Ss}{\mathsf{Sp}}
\newcommand{\Sr}{\mathrm{S}} 

\newcommand{\tone}{\mathbf{1}}
\newcommand{\wX}{\mathbf{X}}

\newcommand{\thh}{\mathbf{h}} 
\newcommand{\tk}{\mathbf{k}}
\newcommand{\tp}{\mathbf{p}} 
\newcommand{\tq}{\mathbf{q}}
\newcommand{\tr}{\mathbf{r}}

\newcommand{\wL}{\mathbf{L}} 
\newcommand{\wI}{\mathbf{I}}
\newcommand{\tLL}{\wI\hspace*{-2pt}\wL} 
\newcommand{\tPi}{\boldsymbol{\Pi}} 
\newcommand{\te}{\mathbf{e}} 
\newcommand{\SC}{\bm{\mathrm{scf}}} 
\newcommand{\Un}{\mathrm{U}}

\newcommand{\Tc}{\mathcal{T}}
\newcommand{\Tcq}{\Tc_q}

\newcommand{\Kc}{\mathcal{K}}
\newcommand{\Kcb}{\overline{\Kc}}

\newcommand{\type}[2]{\mathsf T_{#1}(#2)}
\newcommand{\ordi}[2]{\mathsf O_{#1}(#2)}

\newcommand{\qand}{\quad\text{and}\quad}

\newcommand{\qor}{\quad\text{or}\quad}

\begin{document}


\title{On the {H}adamard product of {H}opf monoids}

\dedicatory{Dedicated to the memory of Jean-Louis Loday}

\author[M.~Aguiar]{Marcelo Aguiar}
\address{Department of Mathematics\\
Texas A\&M University\\
College Station, TX 77843}
\email{maguiar@math.tamu.edu}
\urladdr{http://www.math.tamu.edu/~maguiar}
\thanks{Aguiar supported in part by NSF grant DMS-1001935.}

\author[S.~Mahajan]{Swapneel Mahajan}
\address{Department of Mathematics\\
Indian Institute of Technology Mumbai\\
Powai, Mumbai 400 076\\ India}
\email{swapneel@math.iitb.ac.in}
\urladdr{http://www.math.iitb.ac.in/~swapneel}

\subjclass[2010]{Primary 16T30, 18D35, 20B30; 
Secondary 18D10, 20F55}


\keywords{Species; Hopf monoid; Hadamard product; generating function; Boolean transform.
}

\date{\today}

\begin{abstract}
Combinatorial structures which compose and decompose give rise to Hopf monoids
in Joyal's category of species. The Hadamard product of two Hopf monoids
is another Hopf monoid. We prove two main results regarding freeness of
Hadamard products. The first one states
that if one factor is connected and the other is free as a monoid, 
their Hadamard product is free (and connected). 
The second provides an explicit basis for the Hadamard
product when both factors are free. 

The first main result is obtained by showing the existence of a one-parameter deformation
of the comonoid structure and appealing to a rigidity result of Loday and Ronco
which applies when the parameter is set to zero.
To obtain the second result, we introduce an operation on species which is intertwined
by the free monoid functor with the Hadamard product.
As an application of the first result, we deduce that the dimension sequence of
a connected Hopf monoid satisfies the following condition: except for the first,
all coefficients of the reciprocal of its generating function are nonpositive.
\end{abstract}

\maketitle

\tableofcontents

\section*{Introduction}\label{s:intro}

Combinatorial structures are often equipped with operations which allow to
combine two given structures of a given type into a third and vice versa. 
This leads to the construction of algebraic structures, particularly that of graded
Hopf algebras. When the former are formalized through
the notion of species, which keeps track of the underlying ground set
of the combinatorial structure,
it is possible to construct finer algebraic structures than the latter.
This leads to Hopf monoids in the category of species.
The basic theory of these objects is laid out in~\cite[Part~II]{AguMah:2010},
along with the discussion of several examples.
 Section~\ref{s:hopf} reviews basic material concerning species and Hopf monoids.

Free monoids are the subject of Section~\ref{s:freemonoid}. Just as the tensor algebra
of a vector space carries a canonical structure of Hopf algebra, the free monoid on
a positive species carries one of Hopf monoid. In fact, this structure admits a one
parameter deformation, meaningful even when the parameter $q$ is set to zero. 
The deformation only concerns the comonoid structure;
the monoid structure stays fixed throughout. 
A rigidity result (Theorem~\ref{t:0free}) applies when $q=0$ and makes
this case of particular importance. It states that a connected $0$-Hopf monoid is
necessarily free as a monoid. This is a version of a result for Hopf algebras
of Loday and Ronco~\cite[Theorem~2.6]{LodRon:2006}.

Section~\ref{s:freeness} contains our two main results;
they concern freeness under Hadamard products.
The Hadamard product is a basic operation on species which reflects into the
familiar Hadamard product of the dimension sequences. While there is also a
version of this operation for graded (co)algebras, the case of species is distinguished
by the fact that the Hadamard product of two Hopf monoids is another Hopf monoid
(Proposition~\ref{p:hadamard}). In fact, the Hadamard product 
of a $p$-Hopf monoid $\thh$ and a $q$-Hopf monoid $\tk$ is a $pq$-Hopf monoid $\thh\times\tk$. 
Combining this result with rigidity for  connected $0$-Hopf monoids
we obtain our first main result (Theorem~\ref{t:freeness}).
It states that if $\thh$ is connected and $\tk$ is free as a monoid,
then $\thh\times\tk$ is free as a monoid.
A number of freeness results in the literature (for certain Hopf monoids as well as Hopf algebras)
are consequences of this fact. 
The second main result (Theorem~\ref{t:had-free-monoid})
provides an explicit basis for the Hadamard
product when both factors are free monoids.
To this end, we introduce an operation on species which intertwines
with the Hadamard product via the free monoid functor.

The previous results entail enumerative implications on the dimension sequence of a Hopf monoid.
These are explored in Section~\ref{s:genfun}. They can be conveniently
formulated in terms of the Boolean transform of a sequence (or power series),
since the type generating function of a positive species $\tp$ is the Boolean transform of that of the free monoid on $\tp$.
We deduce that the Boolean transform of the
dimension sequence of a connected Hopf monoid
is nonnegative (Theorem~\ref{t:ordi-hopf}).
This turns out to be stronger than several previously known conditions on the
dimension sequence of a connected Hopf monoid. 
We provide examples of sequences with nonnegative Boolean transform which
do not arise as the dimension sequence of any connected Hopf monoid,
showing that the converse of Theorem~\ref{t:ordi-hopf} does not hold (Proposition~\ref{p:ordi-hopf}).

Appendix~\ref{s:boolean} contains additional information on Boolean transforms;
in particular, Proposition~\ref{p:bool-had} provides
an explicit formula for the Boolean transform of the Hadamard product
of two sequences (in terms of the transforms of the factors). This implies that
the set of real sequences with nonnegative Boolean transform is closed under Hadamard products.

\section{Species and Hopf monoids}\label{s:hopf}

We briefly review Joyal's notion of species~\cite{BerLabLer:1998,Joy:1981}
and of Hopf monoid in the category of species. For more details on the latter,
see~\cite{AguMah:2010}, particularly Chapters 1, 8 and 9.

\subsection{Species and the Cauchy product}\label{ss:cauchy}

Let $\Fset$ denote the category whose objects are finite sets and 
whose morphisms are bijections.
Let $\Kb$ be a field and let $\Vect$ denote the category whose objects are vector spaces over $\Kb$
and whose morphisms are linear maps.

A \emph{(vector) species} is a functor
\[
\Fset \longrightarrow \Vect.
\]
Given a species $\tp$, its value on a finite set $I$ is denoted by $\tp[I]$.
A morphism between species $\tp$ and $\tq$
is a natural transformation between the functors $\tp$ and $\tq$.
Let $\Ss$ denote the category of species.

Given a set $I$ and subsets $S$ and $T$ of $I$, the notation $I=S\sqcup T$
indicates that
\[
I=S\cup T \qand S\cap T = \emptyset.
\]
We say in this case that the ordered pair $(S,T)$ is a \emph{decomposition} of $I$.

Given species $\tp$ and $\tq$, their \emph{Cauchy product} is the species 
$\tp \bdot \tq$ defined on a finite set $I$ by
\begin{equation}\label{e:cau} 
(\tp \bdot \tq)[I] 
 := \bigoplus_{I = S \sqcup T}  \tp[S] \otimes \tq[T].
 \end{equation}
 The direct sum is over all decompositions $(S,T)$ of $I$,
 or equivalently over all subsets $S$ of $I$.
 On a bijection $\sigma: I\to J$, 
 $(\tp \bdot \tq)[\sigma]$ is defined to be the direct sum of the maps
 \[
\tp[S] \otimes \tq[T]\map{\tp[\sigma|_S]\otimes\tp[\sigma|_T]}
\tp[\sigma(S)] \otimes \tq[\sigma(T)]
\]
over all decompositions $(S,T)$ of $I$, where $\sigma|_S$ denotes the restriction of $\sigma$ to $S$.

The operation~\eqref{e:cau} turns $\Ss$ into a monoidal category.
The unit object is the species $\tone$ defined by
\[
\tone[I]  := \begin{cases}
\Kb & \text{if $I$ is empty,} \\
0 & \text{otherwise.}
\end{cases}
\]

Let $q \in \Kb$ be a fixed scalar, possibly zero.
Consider the natural transformation
\[
\beta_q \colon  \tp \bdot \tq \to \tq \bdot \tp
\]
which on a finite set $I$ is the direct sum of the maps
\begin{equation*}
  \tp[S] \otimes \tq[T] \to \tq[T] \otimes \tp[S],
\qquad
x \otimes y \mapsto q^{\abs{S}\abs{T}}   y \otimes x
\end{equation*}
over all decompositions $(S,T)$ of $I$. The notation $\abs{S}$ stands
for the cardinality of the set $S$.

If $q$ is nonzero, then $\beta_q$ is a (strong) braiding
for the monoidal category $(\Ss,\bdot)$.
In this case, the inverse braiding is $\beta_{q^{-1}}$,
and $\beta_q$ is a symmetry if and only if $q=\pm 1$.
The natural transformation $\beta_0$ is a lax braiding
for $(\Ss,\bdot)$.

\subsection{Hopf monoids in species}\label{ss:hopf}

We consider monoids and comonoids in the monoidal category $(\Ss,\bdot)$
and bimonoids and Hopf monoids in the braided monoidal category $(\Ss,\bdot,\beta_q)$.
We refer to the latter as $q$-\emph{bimonoids} and $q$-\emph{Hopf monoids}.
When $q=1$, we speak simply of \emph{bimonoids} and \emph{Hopf monoids}.

The structure of a monoid $\tp$ consists of morphisms of species $\mu:\tp\bdot\tp\to\tp$
and $\iota:\tone\to\tp$ subject to the familiar associative and unital axioms.
In view of~\eqref{e:cau}, the product $\mu$ consists of a collection of linear maps
\[
\mu_{S,T} : \tp[S]\otimes\tp[T] \to \tp[I],
\]
one for each finite set $I$ and each decomposition $(S,T)$ of $I$.
The unit $\iota$ reduces to a linear map
\[
\iota_\emptyset: \Kb \to \tp[\emptyset].
\]
Similarly, the structure of a comonoid $\tq$ consists of linear maps
\[
\Delta_{S,T}: \tq[I]\to\tq[S]\otimes\tq[T]
\qand
\epsilon_\emptyset: \tq[\emptyset] \to \Kb.
\]

Let $I=S\sqcup T=S'\sqcup T'$ be two decompositions of a finite set.
The compatibility axiom for $q$-Hopf monoids
states that the diagram 
\begin{equation}\label{e:compr}
\begin{gathered}
\xymatrix@R+2pc@C-5pt{
\thh[A] \otimes \thh[B] \otimes \thh[C] \otimes \thh[D] \ar[rr]^{\id
\otimes \beta_q \otimes \id} & &
\thh[A] \otimes \thh[C] \otimes \thh[B] \otimes \thh[D] \ar[d]^{\mu_{A,C}
\otimes \mu_{B,D}}\\
\thh[S] \otimes \thh[T] \ar[r]_-{\mu_{S,T}}\ar[u]^{\Delta_{A,B} \otimes
\Delta_{C,D}} & \thh[I] \ar[r]_-{\Delta_{S',T'}} & \thh[S'] \otimes
\thh[T']
}
\end{gathered}
\end{equation}
commutes, where $A=S\cap S'$, $B=S\cap T'$, $C=T\cap S'$, $D=T\cap T'$.
For more details, see~\cite[Sections~8.2 and~8.3]{AguMah:2010}.

\subsection{Connected species and Hopf monoids}\label{ss:con-hopf}

A species $\tp$ is \emph{connected} if $\dim_{\Kb} \tp[\emptyset]=1$.
In a connected monoid, the map $\iota_\emptyset$ is an isomorphism $\Kb\cong\tp[\emptyset]$, and the resulting maps
\[
\tp[I] \cong \tp[I]\otimes\tp[\emptyset] \map{\mu_{I,\emptyset}} \tp[I]
\qand
\tp[I] \cong \tp[\emptyset]\otimes\tp[I] \map{\mu_{\emptyset,I}} \tp[I]
\]
are identities. Thus, to provide a monoid structure on a connected species
it suffices to specify the maps $\mu_{S,T}$ when $S$ and $T$ are nonempty.
A similar remark applies to connected comonoids.

Choosing $S=S'$ and $T=T'$ in~\eqref{e:compr} one obtains that
for a connected $q$-bimonoid $\thh$
the composite
\[
\thh[S]\otimes\thh[T] \map{\mu_{S,T}} \thh[I] \map{\Delta_{S,T}} \thh[S]\otimes\thh[T]
\]
is the identity.

A connected $q$-bimonoid is automatically a $q$-Hopf monoid; see~\cite[Sections~8.4 and~9.1]{AguMah:2010}. The \emph{antipode} of a Hopf monoid will not concern us in this paper.

\subsection{The Hopf monoid of linear orders}\label{ss:linear}

The $q$-Hopf monoid $\wL_q$ is defined as follows. The vector space  $\wL_q[I]$
has for basis the set of linear orders on the finite set $I$. The product and coproduct
are defined by \emph{concatenation} and \emph{restriction}, respectively:
\begin{align*}
\mu_{S,T} : \wL_q[S] \otimes \wL_q[T] & \to \wL_q[I] & \Delta_{S,T}: \wL_q[I] & \to \wL_q[S] \otimes \wL_q[T] \\
l_1 \otimes l_2 & \mapsto l_1 \cdot l_2 & l & \mapsto q^{\area_{S,T}(l)}\, l|_S \otimes l|_T.
\end{align*}Here $l_1 \cdot l_2$ is the linear order on $I$
whose restrictions to $S$ and $T$ are $l_1$ and $l_2$
and in which the elements of $S$ precede the elements of $T$,
and $l|_S$ is the restriction 
of the linear order $l$ on $I$
to the subset $S$. 
The \emph{Schubert cocycle} is
\begin{equation}\label{e:schubert-linear}
\area_{S,T}(l) := \abs{\{(i,j)\in S\times T \mid \text{$i>j$
according to $l$}\}}.
\end{equation}

We write $\wL$ instead of $\wL_1$.
Note that the monoid structure of $\wL_q$ is independent of $q$. 
Thus, $\wL=\wL_q$ as monoids. The comonoid $\wL$ is cocommutative,
but, for $q\neq 1$, $\wL_q$ is not.

\section{The free monoid on a positive species}\label{s:freemonoid}

We review the explicit construction of the free monoid on a positive species,
following~\cite[Section~11.2]{AguMah:2010}. The free monoid carries
a canonical structure of $q$-Hopf monoid. The case $q=0$ is of particular
interest for our purposes, in view of the fact that any connected $0$-Hopf monoid is free
(Theorem~\ref{t:0free} below).

\subsection{Set compositions}\label{ss:compositions}

A \emph{composition} of a finite set $I$ is an ordered sequence $F=(I_1,\dots,I_k)$
of disjoint nonempty subsets of $I$ such that
\[
I  = \bigcup_{i=1}^k I_i.
\]
The subsets $I_i$ are the \emph{blocks}
of $F$. 
We write $F\vDash I$ to indicate that $F$ is a composition of~$I$.

There is only one composition of the empty set
(with no blocks).

Given $I=S\sqcup T$ and compositions $F=(S_1,\dots,S_j)$ of $S$
and $G=(T_1,\dots,T_k)$ of $T$, their \emph{concatenation}
\[
F\cdot G := (S_1,\dots,S_j,T_1,\dots,T_k)
\]
is a composition of $I$.

Given $S\subseteq I$ and a composition $F=(I_1,\dots,I_k)$ of $I$,
we say that $S$ is \emph{$F$-admissible} if for each $i=1,\ldots,k$, either
\[
I_i\subseteq S \qor I_i\cap S=\emptyset.
\]
In this case, we let $i_1<\cdots<i_j$ be the subsequence of $1<\cdots<k$
consisting of those indices $i$ for which $I_i\subseteq S$, and
define the \emph{restriction} of $F$ to $S$ by
\[
F|_S = (I_{i_1},\ldots,I_{i_j}).
\]
It is a composition of $S$.

Given $I=S\sqcup T$ and a composition $F=(I_1,\dots,I_k)$ of $I$,
let
\begin{equation}\label{e:schubert-comp}
\area_{S,T}(F)  := \abs{\{(i,j)\in S\times T \mid \text{$i$ appears in a strictly later block of $F$ than $j$}\}}.
\end{equation}
Alternatively,
\[
\area_{S,T}(F)  = \sum_{1 \leq i < j \leq k} \abs{I_i \cap T}\,\abs{I_j \cap S}.
\]

Still in the preceding situation, note that $S$ is $F$-admissible if and only if $T$ is.
Thus $F|_S$ and $F|_T$ are defined simultaneously.

If the blocks of $F\vDash I$ are singletons, then $F$ amounts to a linear order on $I$.
Concatenation and restriction of set compositions reduce in this case to the
corresponding operations for linear orders (Section~\ref{ss:linear}). In addition,
\eqref{e:schubert-comp} reduces to~\eqref{e:schubert-linear}.

The set of compositions of $I$ is a partial order under \emph{refinement}: we set
$F\leq G$ if each block of $F$ is obtained by merging a number of adjacent blocks of $G$. The composition $(I)$ is the unique minimum element, and linear orders are the maximal elements.

Set compositions of $I$ are in bijection with flags of subsets of $I$ via
\[
(I_1,\ldots,I_k) \mapsto (\emptyset\subset I_1\subset I_1\cup I_2\subset \cdots \subset
I_1\cup \cdots\cup I_k = I).
\]
Refinement of compositions corresponds to inclusion of flags. In this manner the poset of set compositions is
a lower set of the Boolean poset $2^{2^I}$, and hence a meet-semilattice.
The meet operation and concatenation interact as follows:
\begin{equation}\label{e:meet-conc-set}
(F\cdot F')\wedge(G\cdot G') = (F\wedge G)\cdot(F'\wedge G'),
\end{equation}
where $F,G\vDash S$ and $F',G'\vDash T$, $I=S\sqcup T$.

\begin{remark}
Set compositions of $I$ are in bijection with faces of the
\emph{braid arrangement} in $\Rb^I$. Refinement of compositions corresponds to inclusion of faces, meet to
intersection, linear orders to chambers, and $(I)$ to the central face. 
When $S$ and $T$ are nonempty, the statistic $\area_{S,T}(F)$ counts the number of hyperplanes that separate the face $(S,T)$ from $F$. For more details, see~\cite[Chapter~10]{AguMah:2010}.
\end{remark}

\subsection{The free monoid}\label{ss:free}

A species $\tq$ is \emph{positive} if $\tq[\emptyset]=0$.

Given a positive species $\tq$ and a composition $F=(I_1,\ldots,I_k)$ of $I$, write
\begin{equation}\label{e:sp-comp}
\tq(F)  := \tq[I_1] \otimes \dots \otimes \tq[I_k].
\end{equation}
We define a new species $\Tc(\tq)$ by
\[
\Tc(\tq)[I]  := \bigoplus_{F\vDash I} \tq(F).
\]
A bijection $\sigma:I\to J$ transports a composition $F=(I_1,\dots,I_k)$ of $I$
into a composition $\sigma(F):=\bigl(\sigma(I_1),\dots,\sigma(I_k)\bigr)$ of $J$.
The map $\Tc(\tq)[\sigma]:\Tc(\tq)[I]\to\Tc(\tq)[J]$ is the direct sum of the maps
\[
\tq(F) = \tq[I_1] \otimes \dots \otimes \tq[I_k] \map{\tq[\sigma|_{I_1}] \otimes \dots \otimes \tq[\sigma|_{I_k}]} 
\tq[\sigma(I_1)] \otimes \dots \otimes \tq[\sigma(I_k)] = \tq\bigl(\sigma(F)\bigr).
\]

When $F$ is the unique composition of $\emptyset$, we have $\tq(F)=\Kb$.
Thus, the species $\Tc(\tq)$ is connected.

Every nonempty $I$ admits a unique composition with one block; namely, $F=(I)$.
In this case, $\tq(F)=\tq[I]$. This yields an embedding
$
\tq[I] \into \Tc(\tq)[I]
$
and thus an embedding of species 
\[
\eta_\tq: \tq \into \Tc(\tq).
\]
On the empty set, $\eta_\tq$ is (necessarily) zero.

Given $I=S\sqcup T$ and compositions $F\vDash S$ and
$G\vDash T$, we have a canonical isomorphism
\[
 \tq(F)\otimes \tq(G)\cong \tq(F\cdot G)
\]
obtained by concatenating the factors in~\eqref{e:sp-comp}.
The sum of these over all $F\vDash S$ and
$G\vDash T$ yields a map
\[
\mu_{S,T}: \Tc(\tq)[S]\otimes\Tc(\tq)[T] \to \Tc(\tq)[I].
\]
This turns $\Tc(\tq)$ into a monoid. In fact, $\Tc(\tq)$ is the \emph{free}
monoid on the positive species $\tq$, in view of the following result (a slight reformulation of~\cite[Theorem 11.4]{AguMah:2010}).

\begin{theorem}\label{t:dmunivp}
Let $\tp$ be a monoid, $\tq$ a positive species, and
$\zeta\colon \tq \to \tp$ a morphism of species.
Then there exists a unique morphism of monoids
$\hat{\zeta}\colon \Tc(\tq) \to \tp$ such that
\begin{equation*}
\begin{gathered}
\xymatrix@C+20pt{
\Tc(\tq) \ar@{.>}[r]^-{\hat{\zeta}} &  \tp\\
 \tq \ar[ru]_{\zeta}\ar[u]^{\eta_\tq}
}
\end{gathered}
\end{equation*}
commutes.
\end{theorem}

The map $\hat{\zeta}$ is as follows. On the empty set, it is the unit map of $\tp$:
\[
\Tc(\tq)[\emptyset] = \Kb \map{\iota_\emptyset} \tp[\emptyset].
\]
On a nonempty set $I$, it is the sum of the maps
\[
\tq(F) = \tq[I_1]\otimes\cdots\otimes\tq[I_k] \map{\zeta_{I_1}\otimes\cdots\otimes\zeta_{I_k}}
\tp[I_1]\otimes\cdots\otimes\tp[I_k]
\map{\mu_{I_1,\ldots,I_k}} \tp[I],
\]
where $\mu_{I_1,\ldots,I_k}$ denotes an iteration of the product of $\tp$
(well-defined by associativity).

\smallskip 

When there is given an isomorphism of monoids $\tp\cong\Tc(\tq)$, we say that
the positive species $\tq$ is a \emph{basis} of the (free) monoid $\tp$.

\begin{remark}
The free monoid $\Tc(\tq)$ on an arbitrary species $\tq$ exists~\cite[Example~B.29]{AguMah:2010}. One has that $\Tc(\tq)[\emptyset]$ is the free associative unital algebra
on the vector space $\tq[\emptyset]$. Thus, $\Tc(\tq)$ is connected if and only if
$\tq$ is positive. We only consider this case in this paper.
\end{remark}

\subsection{The free monoid as a Hopf monoid}\label{ss:freeHopf}

Let $q\in\Kb$ and $\tq$ a positive species. The species $\Tc(\tq)$ admits a
canonical $q$-Hopf monoid structure, which we denote by $\Tcq(\tq)$, as follows.

As monoids, $\Tcq(\tq)=\Tc(\tq)$. In particular, $\Tcq(\tq)$ and $\Tc(\tq)$ are the same
species. The comonoid structure depends on $q$. Given $I=S\sqcup T$,
the coproduct
\[
\Delta_{S,T}: \Tcq(\tq)[I]\to\Tcq(\tq)[S]\otimes\Tcq(\tq)[T]
\]
is the sum of the maps
\begin{align*}
\tq(F) & \to \tq(F|_S)\otimes \tq(F|_T) \\
x_1\otimes\cdots\otimes x_k & \mapsto 
\begin{cases} 
q^{\area_{S,T}(F)} (x_{i_1}\otimes\cdots\otimes x_{i_j})\otimes (x_{i'_1}\otimes\cdots\otimes x_{i'_k}) & \text{ if $S$ is $F$-admissible,}\\
0 & \text{ otherwise.}
\end{cases}
\end{align*}
Here $F=(I_1,\ldots,I_k)$ and $x_i\in \tq[I_i]$ for each $i$.
In the admissible case, we have written $F|_S=(I_{i_1},\ldots,I_{i_j})$ and
$F|_T=(I_{i'_1},\ldots,I_{i'_k})$.

The preceding turns $\Tcq(\tq)$ into a $q$-bimonoid. Since it is connected,
it is a $q$-Hopf monoid. 

\subsection{Freeness of the Hopf monoid of linear orders}\label{ss:free-linear}

Let $\wX$ be the species defined by
\[
\wX[I]  := \begin{cases}
\Kb & \text{if $I$ is a singleton,} \\
0 & \text{otherwise.}
\end{cases}
\]
It is positive. Note that
\begin{equation}\label{e:XF}
\wX(F) \cong 
\begin{cases}
\Kb & \text{ if all blocks of $F$ are singletons,} \\
 0       & \text{ otherwise.}
\end{cases}
\end{equation}
Since a set composition of $I$ into singletons amounts to a linear order on $I$, we have 
$\Tc(\wX)[I]\cong \wL[I]$ for all finite sets $I$.
This gives rise to a canonical isomorphism of species 
\[
\Tc(\wX) \cong \wL.
\]
Moreover, the closing remarks in Section~\ref{ss:compositions}
imply that this is an isomorphism of $q$-Hopf monoids
\[
\Tcq(\wX) \cong \wL_q.
\]
In particular, $\wL$ is the free monoid on the species $\wX$.

\subsection{Loday-Ronco freeness for $0$-Hopf monoids}\label{ss:0free}

The $0$-Hopf monoid 
$\Tc_0(\tq)$ has the same underlying species and the same product
as the Hopf monoid $\Tc(\tq)$ (Section~\ref{ss:free}).
We now discuss the coproduct, by setting $q=0$ in the description
of Section~\ref{ss:freeHopf}.
Fix a decomposition $I=S\sqcup T$.
The compositions $F\vDash I$ that contribute to $\Delta_{S,T}$
are those for which
$S$ is $F$-admissible and in addition $\area_{S,T}(F)=0$. This happens if and only if
\[
F = F|_S \cdot F|_T.
\]
When $S,T\neq\emptyset$, the preceding is in turn equivalent to
\begin{equation}\label{e:0free}
(S,T)\leq F.
\end{equation}
Therefore, the coproduct $\Delta_{S,T}$ of $\Tc_0(\tq)$
is the direct sum over all $F\vDash I$ of the above form of the maps 
\begin{align*}
\tq(F) & \to \tq(F|_S)\otimes \tq(F|_T) \\
x_1\otimes\cdots\otimes x_k & \mapsto  
 (x_{1}\otimes\cdots\otimes x_{j})\otimes (x_{j+1}\otimes\cdots\otimes x_{k}).
 \end{align*}
Here $F=(I_1,\ldots,I_k)$, $S=I_1\cup\cdots\cup I_j$, and $T=I_{j+1}\cup\cdots\cup I_k$.

\begin{theorem}\label{t:0free}
Let $\thh$ be a connected $0$-Hopf monoid. Then there exist a positive species
$\tq$ and an isomorphism of $0$-Hopf monoids
\[
\thh \iso \Tc_0(\tq).
\]  
\end{theorem}

The species $\tq$ can be obtained as the \emph{primitive part} of $\thh$.

There is a parallel result for connected graded $0$-Hopf algebras which is
due to Loday and Ronco~\cite[Theorem~2.6]{LodRon:2006}.
An adaptation of their proof yields the result for connected $0$-Hopf monoids;
the complete details are given in~\cite[Theorem~11.49]{AguMah:2010}.

\begin{remark}
Theorem~\ref{t:0free} states that any connected $0$-Hopf monoid is free as a
monoid. It is also true that it is \emph{cofree} as a comonoid; in addition, if $\tq$ is finite-dimensional,
then the $0$-Hopf monoid $\Tc_0(\tq)$ is \emph{self-dual}. See~\cite[Section~11.10.3]{AguMah:2010} for more details.
\end{remark}

\section{Freeness under Hadamard products}\label{s:freeness}

The Hadamard product of two Hopf monoids is another Hopf monoid. 
We review this construction and
we prove in Theorem~\ref{t:freeness} 
that if one of the Hopf monoids is free as a monoid, then
the Hadamard product is also free as a monoid 
(provided the other Hopf monoid is connected).
We introduce an operation on positive species which
allows us to describe a basis for the Hadamard
product of two free monoids in terms of bases of the factors (Theorem~\ref{t:had-free-monoid}).

\subsection{The Hadamard product of Hopf monoids}\label{ss:hadamard}

The \emph{Hadamard product} of two species $\tp$ and $\tq$ is the species
$\tp\times\tq$ defined on a finite set $I$ by
\[
(\tp\times\tq)[I] := \tp[I]\otimes\tq[I],
\]
and on bijections similarly. 

If $\tp$ and $\tq$ are connected, then so is $\tp\times\tq$.

\begin{proposition}\label{p:hadamard}
Let $p,q\in\Kb$ be arbitrary scalars.
If $\thh$ is a $p$-bimonoid and $\tk$ is a $q$-bimonoid,
then $\thh\times\tk$ is a $pq$-bimonoid.
\end{proposition}
The proof is given in~\cite[Corollary~9.6]{AguMah:2010}.
The corresponding statement for Hopf monoids holds as well.

The product of $\thh\times\tk$ is defined by
\[
\xymatrix@C+11pt{
(\thh\times\tk)[S]\otimes(\thh\times\tk)[T] \ar@{.>}[rr]^-{\mu_{S,T}} \ar@{=}[d] & & 
(\thh\times\tk)[I] \ar@{=}[d]\\
(\thh[S]\otimes\tk[S])\otimes(\thh[T]\otimes\tk[T]) \ar[r]_-{\cong} & 
(\thh[S]\otimes\thh[T])\otimes(\tk[S]\otimes\tk[T]) \ar[r]_-{\mu_{S,T}\otimes\mu_{S,T}} &
\thh[I] \otimes\tk[I] 
}
\]
where the first map on the bottom simply switches the middle tensor factors.
The coproduct is defined similarly.

In particular, if $\thh$ and $\tk$ are bimonoids ($p=q=1$), then so is $\thh\times\tk$.

\begin{remark}
There is a parallel between the notions of species on the one hand,
and of graded vector spaces on the other. This extends to a parallel between
Hopf monoids in species and graded Hopf algebras. These topics are
studied in detail in~\cite[Part III]{AguMah:2010}.

The Hadamard product of graded vector spaces can be defined, but does not enjoy
the same formal properties of that for species. In particular, the Hadamard product
of two graded bialgebras carries natural algebra and coalgebra structures, but these
are not compatible in general; see~\cite[Remark~8.65]{AguMah:2010}.
For this reason, our main result (Theorem~\ref{t:freeness} below) does not
possess an analogue for graded bialgebras.
\end{remark}

\subsection{Freeness under Hadamard products}\label{ss:freeness}

The following is our main result. Let $p$ and $q\in\Kb$ be arbitrary scalars.

\begin{theorem}\label{t:freeness}
Let $\thh$ be a connected $p$-Hopf monoid. 
Let $\tk$ be a $q$-Hopf monoid that is free as a monoid.
Then $\thh\times\tk$ is a connected $pq$-Hopf monoid that is free as a monoid.
\end{theorem}
\begin{proof}
Since $\thh$ and $\tk$ are connected (the latter by freeness), so is $\thh\times\tk$.
We then know from Proposition~\ref{p:hadamard} that $\thh\times\tk$ is a connected $pq$-Hopf monoid. Now, as monoids, we have
\[
\tk\cong \Tc_q(\tq) = \Tc_0(\tq)
\]
for some positive species $\tq$. Hence, as monoids,
\[
\thh\times\tk \cong \thh \times \Tc_0(\tq).
\]
But the latter is a $0$-Hopf monoid by Proposition~\ref{p:hadamard}, and hence
free as a monoid by Theorem~\ref{t:0free}.
\end{proof}

\begin{corollary}\label{c:freeness}
Let $\thh$ be a connected $p$-Hopf monoid. Then $\thh\times\wL_q$ is free as a monoid.
\end{corollary}
\begin{proof}
This is a special case of Theorem~\ref{t:freeness}, since as discussed in Section~\ref{ss:free-linear}, $\wL_q\cong\Tc_q(\wX)$.
\end{proof}


\begin{example}\label{eg:LL}
The Hopf monoid $\tLL_q$ of \emph{pairs of linear orders} is studied in~\cite[Section~12.3]{AguMah:2010}. There is an isomorphism of $q$-Hopf monoids
\[
 \tLL_q \cong \wL^{*} \times \wL_q.
 \]
Corollary~\ref{c:freeness} implies that $\tLL_q$ is free as a monoid.
This result was obtained by different means in~\cite[Section~12.3]{AguMah:2010}.
It implies the fact that the Hopf algebra of permutations of Malvenuto
and Reutenauer~\cite{MalReu:1995} is free as an algebra, a result known from~\cite{PoiReu:1995}. See Section~\ref{ss:livernet} below for more comments regarding 
connections between Hopf monoids and Hopf algebras.
\end{example}

\begin{example}\label{eg:super}
The Hopf monoid $\SC(\Un)$ of \emph{superclass functions} 
on unitriangular matrices with entries in $\Fb_2$ is studied in~\cite{ABT:2011}.
There is an isomorphism of Hopf monoids
\[
\SC(\Un) \cong \tPi\times\wL,
\]
where $\tPi$ is the Hopf monoid of \emph{set partitions} of~\cite[Section~12.6]{AguMah:2010}. It follows that $\SC(\Un)$ is free as a monoid. This result was obtained
by different means in~\cite[Proposition~17]{ABT:2011}. It implies the fact that
the Hopf algebra of \emph{symmetric functions in noncommuting variables}~\cite{RosSag:2006} is free
as an algebra, a result known from~\cite{Wol:1936}.
\end{example}

\subsection{Livernet freeness for certain Hopf algebras}\label{ss:livernet}

It is possible to associate a number of graded Hopf algebras to
a given Hopf monoid $\thh$. This is the subject of~\cite[Chapter~15]{AguMah:2010}.
In particular, there are two graded Hopf algebras $\Kc(\thh)$
and $\Kcb(\thh)$ related by a canonical surjective morphism
\[
\Kc(\thh) \onto \Kcb(\thh).
\]
Moreover, for any Hopf monoid $\thh$, there is a canonical isomorphism of graded Hopf algebras~\cite[Theorem~15.13]{AguMah:2010}
\[
\Kcb(\wL\times\thh) \cong \Kc(\thh).
\]

The functor $\Kcb$ preserves a number of properties, including freeness:
if $\thh$ is free as a monoid, then $\Kcb(\thh)$ is free as an algebra~\cite[Proposition~18.7]{AguMah:2010}.

Combining these remarks with Corollary~\ref{c:freeness} we deduce that for any
connected Hopf monoid $\thh$, the algebra $\Kc(\thh)$ is free. This result is
due to Livernet~\cite[Theorem~4.2.2]{Liv:2010}. A proof similar to the one above
is given in~\cite[Section~16.1.7]{AguMah:2010}.

As an example, for $\thh=\tLL$ we obtain that the Hopf algebra of pairs of permutations
is free as an algebra, a result known from~\cite[Theorem 7.5.4]{AguMah:2006}.


\subsection{The Hadamard product of free monoids}\label{ss:had-free}

Given positive species $\tp$ and $\tq$,
define a new positive species $\tp \star \tq$ by
\begin{equation}\label{e:star}
(\tp \star \tq)[I] :=
\bigoplus_{\substack{F,G\vDash I\\F\wedge G=(I)}} \tp(F)\otimes\tq(G).
\end{equation}
The sum is over all pairs $(F,G)$ of compositions of $I$
such that $F\wedge G=(I)$. We are employing notation~\eqref{e:sp-comp}.

\begin{lemma}\label{l:star-prop}
For any composition $H\vDash I$, there is a canonical isomorphism of vector spaces
\begin{equation}\label{e:star-prop}
(\tp \star \tq)(H) \cong \bigoplus_{\substack{F,G\vDash I\\F\wedge G=H}}\tp(F)\otimes\tq(G)
\end{equation}
given by rearrangement of the tensor factors.
\end{lemma}
\begin{proof}
Let us say that a function $f$ on set compositions with values on vector spaces is
\emph{multiplicative} if $f(H_1\cdot H_2) \cong f(H_1)\otimes f(H_2)$ for all
$H_1\vDash I_1$, $H_2\vDash I_2$, $I=I_1\sqcup I_2$. Such functions are uniquely determined
by their values on the compositions of the form $(I)$. The isomorphism~\eqref{e:star-prop} holds when $H=(I)$ by definition~\eqref{e:star}. It thus suffices to check that both
sides are multiplicative.

The left hand side of~\eqref{e:star-prop} is multiplicative in view of~\eqref{e:sp-comp}. 

If, for $i=1,2$, $F_i,G_i\vDash I_i$ are such that $F_i\wedge G_i=H_i$, then
\[
(F_1\cdot F_2)\wedge(G_1\cdot G_2) = H_1\cdot H_2
\]
by~\eqref{e:meet-conc-set}. Moreover, if $F,G\vDash I_1\sqcup I_2$ are such that 
$F\wedge G=H_1\cdot H_2$, then $F=F_1\cdot F_2$ and $G=G_1\cdot G_2$ for
unique $F_i,G_i$ as above.
This implies the multiplicativity of the right hand side.
\end{proof}

We show that the operation~\eqref{e:star} is associative.

\begin{proposition}\label{p:star-asso}
For any positive species $\tp$, $\tq$ and $\tr$,
there is a canonical isomorphism
\begin{equation}
(\tp \star \tq) \star \tr \iso \tp \star (\tq \star \tr).
\end{equation}
\end{proposition}
\begin{proof} Define
\[
(\tp \star \tq \star \tr)[I] := \bigoplus_{\substack{F,G,H\vDash I,\\F\wedge G\wedge H=(I)}} \tp(F)\otimes\tq(G)\otimes\tr(H).
\]
We make use of the isomorphism~\eqref{e:star-prop} to build the following.
\begin{align*}
\bigl(\tp \star (\tq \star \tr)\bigr)[I] & =
\bigoplus_{\substack{F,K\vDash I\\ F\wedge K=(I)}} \tp(F)\otimes (\tq\star \tr)(K)\\
& \cong \bigoplus_{\substack{F,K\vDash I,\\F\wedge K=(I)}}\bigoplus_{\substack{G,H\vDash I,\\G\wedge H=K}} \tp(F)\otimes\tq(G)\otimes\tr(H)\\
& = \bigoplus_{\substack{F,G,H\vDash I,\\F\wedge G\wedge H=(I)}} \tp(F)\otimes\tq(G)\otimes\tr(H) = (\tp \star \tq \star \tr)[I]
\end{align*}
The space $\bigl((\tp \star \tq) \star \tr\bigr)[I]$ can be identified with
$(\tp \star \tq \star \tr)[I]$ in a similar manner.
\end{proof}

There is also an evident natural isomorphism
\begin{equation}
\tp \star \tq \iso \tq \star \tp.
\end{equation}
Thus, $\star$ defines a (nonunital) symmetric monoidal structure on the category of positive species.

Our present interest in the operation $\star$ stems from the following result, which
provides an explicit description for the basis of a Hadamard product of two free monoids
in terms of bases of the factors.

\begin{theorem}\label{t:had-free-monoid}
For any positive species $\tp$ and $\tq$,
there is a natural isomorphism of monoids
\begin{equation}\label{e:had-free-monoid}
\Tc(\tp \star \tq) \cong \Tc(\tp)\times\Tc(\tq).
\end{equation}
\end{theorem}
\begin{proof}
We calculate using~\eqref{e:star-prop}.
\begin{multline*}
\Tc(\tp \star \tq)[I] =
\bigoplus_{H\vDash I} (\tp \star \tq)(H)  \cong
\bigoplus_{H\vDash I} \bigoplus_{\substack{F,G\vDash I\\F\wedge G=H}} \tp(F)\otimes\tq(G)\\
= \bigoplus_{F,G\vDash I} \tp(F)\otimes\tq(G)
= \Tc(\tp)[I]\otimes\Tc(\tq)[I] =
\bigl(\Tc(\tp)\times\Tc(\tq)\bigr)[I].
\end{multline*}
The fact that this isomorphism preserves products follows from~\eqref{e:meet-conc-set}. 
\end{proof}

\begin{example}\label{eg:basisLL}
Since $\wX$ is a basis for $\wL$,
\eqref{e:had-free-monoid} implies that $\wX\star\wX$ is a basis for $\wL\times\wL$.
From~\eqref{e:XF} we obtain that
\[
\{(C,D) \mid C\wedge D = (I)\}.
\]
is a linear basis for $(\wX\star\wX)[I]$. (The linear orders $C$ and $D$ are viewed as set compositions into singletons.)
For related results, see~\cite[Section 12.3.6]{AguMah:2010}.
\end{example}

Recall that, for each scalar $q\in\Kb$,
any free monoid $\Tc(\tp)$ is endowed with a canonical
comonoid structure and the resulting $q$-Hopf monoid is denoted $\Tcq(\tp)$ 
(Section~\ref{ss:freeHopf}).
It turns out that, when $q=0$,~\eqref{e:had-free-monoid} is in fact an isomorphism
of $0$-Hopf monoids, as we now prove. The proof below also shows that~\eqref{e:had-free-monoid} is not an isomorphism of comonoids for $q\neq 0$.

\begin{proposition}\label{p:had-free-monoid}
The map~\eqref{e:had-free-monoid}
is an isomorphism of $0$-Hopf monoids
\[
\Tc_0(\tp \star \tq) \cong \Tc_0(\tp)\times\Tc_0(\tq).
\]
\end{proposition}
\begin{proof}
In order to prove that coproducts are preserved it  suffices to check that they
agree on the basis $\tp\star\tq$ of $\Tc(\tp \star \tq)$ and on its image in $\Tc(\tp)\times\Tc(\tq)$. The image of $(\tp \star \tq)[I]$ is the direct sum of the spaces $\tp(F)\otimes\tq(G)$ over those $F,G\vDash I$ such that $F\wedge G =(I)$.
Choose $S,T\neq\emptyset$ such that $I=S\sqcup T$. 
In view of the definition of the coproduct on a free monoid (Section~\ref{ss:freeHopf}),
the coproduct $\Delta_{S,T}$
of $\Tc_q(\tp \star \tq)$ is zero on $(\tp \star \tq)[I]$. (This holds for any $q\in\Kb$.)
On the other hand, from~\eqref{e:0free} we have that the coproduct of
 $\Tc_0(\tp)\times\Tc_0(\tq)$ on $\tp(F)\otimes\tq(G)$ is also zero, unless both
 \[
 (S,T)\leq F \qand (S,T)\leq G.
 \]
 Since this is forbidden by the assumption $F\wedge G =(I)$, the coproducts agree.
 \end{proof}


\section{The dimension sequence of a connected Hopf monoid}
\label{s:genfun}

We now derive a somewhat surprising application of Theorem~\ref{t:freeness}. 
It states that the reciprocal of the ordinary generating function of a connected Hopf monoid
has nonpositive (integer) coefficients (Theorem~\ref{t:ordi-hopf} below).
We compare this result with other previously known conditions satisfied by
the dimension sequence of a connected Hopf monoid.

\subsection{Coinvariants}\label{ss:coinvariants}

Let $G$ be a group and $V$ a $\Kb G$-module. The space of \emph{coinvariants} $V_{G}$ is
the quotient of $V$ by the $\Kb$-subspace spanned by the elements of the form
\[
v-g\cdot v
\]
for $v\in V$, $g\in G$. If $V$ is a free $\Kb G$-module, then 
\[
\dim_{\Kb} V_G = \rank_{\Kb G} V.
\]

Let $V$ and $W$ be $\Kb G$-modules. Let $U_1$ be the space $V\otimes W$
with \emph{diagonal} $G$-action:
\[
g\cdot (v\otimes w) := (g\cdot v)\otimes(g\cdot w).
\]
Let $U_2$ be the same space but with the following $G$-action:
\[
g\cdot (v\otimes w) := v\otimes (g\cdot w).
\]
The following is a well-known basic fact.

\begin{lemma}\label{l:diagonal}
If $W$ is free as a $\Kb G$-module, then $U_1\cong U_2$. In particular, 
\[
\dim_{\Kb} (U_1)_G = (\dim_{\Kb} V)(\dim_{\Kb} W_G).
\]
\end{lemma}
\begin{proof}
We may assume $W=\Kb G$. In this case, the map
\[
U_1\to U_2, \quad
v\otimes g \mapsto (g^{-1}\cdot v) \otimes g
\]
is an isomorphism of $\Kb G$-modules. The second assertion follows because $U_2$
is a free module of rank equal to $(\dim_{\Kb} V)(\rank_{\Kb G} W)$.
\end{proof}

\subsection{The type generating function}\label{ss:type}

Let $\tp$ be a species. We write $\tp[n]$ for the space $\tp[\{1,\ldots,n\}]$.
The symmetric group $\Sr_n$ acts on $\tp[n]$ by
\[
\sigma\cdot x := \tp[\sigma](x)
\]
for $\sigma\in\Sr_n$, $x\in\tp[n]$. For example,
\[
\wL[n] \cong \Kb\Sr_n
\]
as $\Kb\Sr_n$-modules.

From now on, we assume that all species $\tp$ are \emph{finite-dimensional}.
This means that for each $n\geq 0$ the space $\tp[n]$ is finite-dimensional. 
The \emph{type generating function} of $\tp$ is the power series
\[
\type{\tp}{x} : = \sum_{n\geq 0} \dim_{\Kb} \tp[n]_{S_n}\, x^n.
\]

For example,
\[
\type{\wL}{x} = \sum_{n\geq 0} x^n=\frac{1}{1-x}.
\]
More generally, for any positive species $\tq$,
\begin{equation}\label{e:type-free}
\type{\Tc(\tq)}{x} = \frac{1}{1-\type{\tq}{x}}.
\end{equation}
This follows by a direct calculation or from~\cite[Theorem~1.4.2.b]{BerLabLer:1998}.

Let $\tp$ be a free monoid. It follows from~\eqref{e:type-free} that
\begin{equation}\label{e:type-free2}
1-\frac{1}{\type{\tp}{x}} \in \Nb\llb x\rrb.
\end{equation}
In other words, the reciprocal of the type generating function of a free
monoid has nonpositive integer coefficients (except for the first, which is $1$).

\subsection{Generating functions for Hadamard products}\label{ss:gen-had}

The type generating function of a Hadamard product $\tp\times\tq$ is in general not
determined by those of the factors. (It is however determined by their \emph{cycle indices}; see~\cite[Proposition~2.1.7.b]{BerLabLer:1998}.) 

The \emph{ordinary generating function} of a species $\tp$ is
\[
\ordi{\tp}{x} : = \sum_{n\geq 0} \dim_{\Kb} \tp[n]\, x^n.
\]

The Hadamard product of power series is defined by
\[
\bigl(\sum_{n\geq 0} a_n x^n\bigr)\times\bigl(\sum_{n\geq 0} b_n x^n\bigr) :=
\sum_{n\geq 0} a_nb_n x^n.
\]

\begin{proposition}\label{p:free-had}
Let $\tp$ be an arbitrary species and $\tq$ a species for which $\tq[n]$ is a free $\Kb\Sr_n$-module for every $n\geq 0$. Then
\begin{equation}\label{e:free-had}
\type{\tp\times \tq}{x} = \ordi{\tp}{x}\times \type{\tq}{x}. 
\end{equation}
\end{proposition}
\begin{proof}
In view of Lemma~\ref{l:diagonal}, we have
\[
\dim_{\Kb} \bigl((\tp\times\tq)[n]\bigr)_{\Sr_n} = (\dim_{\Kb} \tp[n])(\dim_{\Kb} \tq[n]_{\Sr_n})
\]
from which the result follows.
\end{proof}

Since $\type{\wL}{x}$ is the unit for the Hadamard product of power series, we have from~\eqref{e:free-had} that
\begin{equation}\label{e:type-hadamard}
\type{\tp\times \wL}{x} = \ordi{\tp}{x}.
\end{equation}
More generally, for any positive species $\tq$,
\begin{equation}\label{e:type-hadamard2}
\type{\tp\times \Tc(\tq)}{x} = \ordi{\tp}{x}\times  \frac{1}{1-\type{\tq}{x}}.
\end{equation}
This follows from~\eqref{e:type-free} and~\eqref{e:free-had};
the $\Kb\Sr_n$-module $\Tc(\tq)[n]$ is free by~\cite[Lemma~B.18]{AguMah:2010}.

\subsection{The ordinary generating function of a connected Hopf monoid}\label{ss:ordi-hopf}

 Let $\thh$ be a connected $q$-Hopf monoid. 
 By Corollary~\ref{c:freeness}, $\thh\times\wL$ is a free monoid. 
 Let $\tq$ be a basis. Thus, $\tq$ is a positive species such that
\[
\thh\times\wL\cong \Tc(\tq)
\]
as monoids. 

 \begin{proposition}\label{p:ordi-type}
 In the above situation,
 \begin{equation}\label{e:ordi-type}
 \ordi{\thh}{x} = \frac{1}{1-\type{\tq}{x}}.
 \end{equation}
 \end{proposition}
 \begin{proof}
 We have, by~\eqref{e:type-free} and~\eqref{e:type-hadamard},
 \[
 \ordi{\thh}{x} =
 \type{\thh\times\wL}{x}=\type{\Tc(\tq)}{x}=\frac{1}{1-\type{\tq}{x}}.  \qedhere
 \]
 \end{proof}

\begin{theorem}\label{t:ordi-hopf}
Let $\thh$ be a connected $q$-Hopf monoid. Then
\begin{equation}\label{e:ordi-hopf}
1-\frac{1}{\ordi{\thh}{x}} \in \Nb\llb x\rrb.
\end{equation}
\end{theorem}
\begin{proof}
From~\eqref{e:ordi-type} we deduce
\[
1-\frac{1}{\ordi{\thh}{x}} = \type{\tq}{x}
 \in \Nb\llb x\rrb. \qedhere
 \]
\end{proof}

In the terminology of Section~\ref{s:boolean} below, Theorem~\ref{t:ordi-hopf} states that the Boolean transform of the
dimension sequence of a connected $q$-Hopf monoid
is nonnegative; see~\eqref{e:bool-tran}. Proposition~\ref{p:ordi-type} states more
precisely that the Boolean transform of the ordinary generating function of $\thh$ is the
type generating function of $\tq$.

\begin{example}\label{eg:globaldes}
We have
\[
1-\frac{1}{\sum_{n\geq 0} n! x^n} =  x + x^2 + 3x^3 + 13x^4 + 71x^5+ 461 x^6 + \cdots.
\]
The Boolean transform $b_n$ of the dimension sequence of $\wL$
admits the following description. Say that a linear order on the set $[n]$ is
\emph{decomposable} if it is the concatenation of a linear order on $[i]$
and a linear order on $[n]\setminus [i]$ for some $i$ such that $1\leq i<n$.
Every linear order is the concatenation of a unique sequence of indecomposable ones.
It then follows from~\eqref{e:bool-tran3} that $b_n$ is the number of indecomposable linear orders on $[n]$.
The sequence $b_n$ is~\cite[A003319]{Slo:oeis}. 
\end{example}

\begin{example}\label{eg:atomic}
A partition $X$ of the set $[n]$ is \emph{atomic} if $[i]$
is not a union of blocks of $X$ for any $i$ such that $1\leq i<n$.
The dimension sequence of the Hopf monoid $\tPi$ is the sequence of Bell numbers,
and its Boolean transform counts the number of atomic partitions of $[n]$.
The latter is sequence~\cite[A074664]{Slo:oeis}.
\end{example}

Let $a_n:=\dim_{\Kb} \thh[n]$. The conditions imposed by~\eqref{e:ordi-hopf} on the first terms of this sequence are as follows.
\begin{gather*}
a_1^2 \leq a_2,\\
2a_1a_2-a_1^3 \leq a_3,\\
2a_1a_3 - 3a_1^2a_2+a_2^2 + a_1^4 \leq a_4.
\end{gather*}

\begin{example}\label{eg:elements}
Suppose that the sequence starts with
\[
a_1=1,\quad a_2=2,\qand a_3=3.
\]
The third inequality above then implies $a_4\geq 5$. It follows that the species $\te$
\emph{of elements} (for which $\dim_{\Kb} \te[n]=n$) does not carry a bimonoid structure.
This result was obtained by different means in~\cite[Example~3.5]{AL:2012}. 
\end{example}

The calculation of Example~\ref{eg:globaldes} can be generalized to all free
monoids in place of $\wL$. To this end, let us say that a composition $F$ of the set
$[n]$ is \emph{decomposable} if $F=F_1\cdot F_2$ for some $F_1\vDash [i]$,
$F_2\vDash [n]\setminus[i]$, and some $i$ such that $1\leq i<n$.

\begin{proposition}\label{p:globaldes}
For any positive species $\tp$, the Boolean transform of the dimension sequence of 
the Hopf monoid $\Tc(\tp)$ is given by
\[
b_n =  \sum_{\substack{F\vDash [n] \\ \text{$F$ indecomposable}}} \dim_\Kb \tp(F).
\]
\end{proposition}
\begin{proof}
We have from~\eqref{e:had-free-monoid} that
\[
\Tc(\tp\star\wX) \cong \Tc(\tp)\times \Tc(\wX) \cong \Tc(\tp)\times\wL.
\]
Hence, by~\eqref{e:ordi-type},
\[
\ordi{\Tc(\tp)}{x} = \frac{1}{1-\type{\tp\star\wX}{x}}.
\]
Thus, $\type{\tp\star\wX}{x}$ is the Boolean transform of $\ordi{\Tc(\tp)}{x}$,
and hence $b_n = \dim_{\Kb} \bigl((\tp\star\wX)[n]\bigr)_{\Sr_n}$.

From~\eqref{e:XF} and~\eqref{e:star} we have that
\[
(\tp\star\wX)[I] = \bigoplus_{(F,C):\,F\wedge C=(I)} \tp(F)
\]
where $F$ varies over set compositions and $C$ varies over linear orders on $I$.
It follows that $(\tp\star\wX)[n]$ is a free $\Kb\Sr_n$-module with
$\Sr_n$-coinvariants equal to the space 
\[
\bigoplus_{F\vDash [n],\, F\wedge C_n=([n])} \tp(F)
\]
where $C_n$ denotes the canonical linear order on $[n]$. The result follows since
$F\wedge C_n=([n])$ if and only if $F$ is indecomposable.

(Alternatively, we may prove this result by appealing to~\eqref{e:bool-tran3} as in Example~\ref{eg:globaldes}.)
\end{proof}


Let $\thh$ and $\tk$ be connected Hopf monoids.
The Boolean transform of the dimension sequence of $\thh\times\tk$
can be explicitly described 
in terms of the Boolean transforms of the dimension sequences of $\thh$ and $\tk$;
see Proposition~\ref{p:bool-had} below.

For example, let $b_n$ be the Boolean transform
of the dimension sequence of $\tLL$ (Example~\ref{eg:LL}).
This is sequence~\cite[A113871]{Slo:oeis} and its
first few terms are $1,3,29,499$.
Recalling that $\tLL\cong \wL^*\times\wL$ and
employing~\eqref{e:bool-had} we readily obtain that $b_n$
counts the number of pairs $(l,m)$ of linear orders on $[n]$
such that $\alpha \wedge \beta = (n)$,
where the sequence of indecomposables of $l$ has size $\alpha$ and
that of $m$ has size $\beta$.

\begin{remark}
Theorem~\ref{t:ordi-hopf} states that if $\thh$ is a connected $q$-Hopf monoid, then 
the Boolean transform of $\ordi{\thh}{x}$ is nonnegative. This was deduced by considering the Hadamard product of $\thh$ with $\wL$.
One may also consider the Hadamard product of $\thh$ with an arbitrary free Hopf monoid.
Then, using Theorem~\ref{t:freeness}
together with~\eqref{e:type-free2} and~\eqref{e:type-hadamard2}, one obtains that
for any series $A(x)\in \Nb\llb x\rrb$ with nonnegative Boolean transform, the Hadamard product $\ordi{\thh}{x}\times A(x)$ also has nonnegative Boolean transform.
However, this does not impose any additional conditions on  $\ordi{\thh}{x}$,
in view of Corollary~\ref{c:submonoid}.
\end{remark}

\subsection{Non-attainable sequences}\label{ss:non}

The question arises whether condition~\eqref{e:ordi-hopf} characterizes
the dimension sequence of a connected Hopf monoid. In other words,
given a sequence of nonnegative integers $b_n$, $n\geq 1$, 
is there a connected $q$-Hopf monoid $\thh$ such that
\begin{equation}\label{e:ordi-hopf2}
1-\frac{1}{\ordi{\thh}{x}} = \sum_{n\geq 1} b_n x^n
\end{equation}
holds? In other words, the question is
whether $b_n$ is the Boolean transform of the
dimension sequence of a connected $q$-Hopf monoid. The answer is negative,
as the following result shows.

\begin{proposition}\label{p:ordi-hopf}
Consider the sequence defined by
\[
b_n:= \begin{cases}
 1 & \text{ if $n=2$,} \\
 0        & \text{ otherwise.}
\end{cases}
\]
Then there is no connected $q$-bimonoid $\thh$ for which~\eqref{e:ordi-hopf2} holds,
regardless of $q$.
\end{proposition}
\begin{proof}
Suppose such $\thh$ exists; let $a_n$ be its dimension sequence.
Then $b_n$ is the Boolean transform of $a_n$, and~\eqref{e:bool-tran3}
implies that
\[
a_n:= \begin{cases}
 1 & \text{ if $n$ is even,} \\
 0        & \text{ if $n$ is odd.}
\end{cases}
\]

Recall from Section~\ref{ss:con-hopf} that for any decomposition $I=S\sqcup T$,
the composite $\Delta_{S,T}\mu_{S,T}$ is the identity. 
It follows in the present situation that
$\mu_{S,T}$ and $\Delta_{S,T}$ are inverse 
whenever $S$ and $T$ are of even cardinality.
Now let
\[
I=\{a,b,c,d\},\quad S=\{a,b\},\quad T=\{c,d\},\quad S'=\{a,c\},\qand T'=\{b,d\}
\]
and consider the commutative diagram~\eqref{e:compr}.
The bottom horizontal composite in this diagram is an isomorphism between one-dimensional vector spaces,
while the composite obtained by going up, across and down is the zero map.
This is a contradiction.
\end{proof}

Let $k$ be a positive integer and define, for $n\geq 1$,
\[
b^{(k)}_n:= \begin{cases}
 1 & \text{ if $n=k$,} \\
 0        & \text{ otherwise.}
\end{cases}
\]
The inverse Boolean transform of $b^{(k)}_n$ is
\[
a^{(k)}_n:= \begin{cases}
 1 & \text{ if $n\equiv 0\!\!\mod k$,} \\
 0        & \text{ otherwise.}
\end{cases}
\]
An argument similar to that in Proposition~\ref{p:ordi-hopf} shows that, if $k\geq 2$,
there is no connected $q$-Hopf monoid with dimension sequence $a^{(k)}_n$.
(The exponential Hopf monoid~\cite[Example~8.15]{AguMah:2010} has dimension sequence $a^{(1)}_n$.)

\subsection{Comparison with previously known conditions}\label{ss:comparison}

The paper~\cite{AL:2012} provides various sets of conditions that the dimension sequence $a_n$
of a connected Hopf monoid must satisfy. For instance,~\cite[Proposition~4.1]{AL:2012}
states that
\begin{equation}\label{e:submult}
a_ia_j \leq a_n
\end{equation}
for $n=i+j$ and every $i,j\geq 1$. In addition, the coefficients of the power series
\begin{equation}\label{e:exp}
\frac{1 + \sum_{n\geq 1} a_n x^n}{1 + \sum_{n\geq 1} \frac{a_n}{n!} x^n}
\end{equation}
are nonnegative~\cite[Corollary 3.3]{AL:2012}, and~\cite[Equation~(3.2)]{AL:2012} states that
\begin{equation}\label{e:addcond}
a_3\geq 3a_2a_1-2a_1^3.
\end{equation}
We proceed to compare these conditions with those imposed by Theorem~\ref{t:ordi-hopf}.

The inequalities~\eqref{e:submult} are implied by Theorem~\ref{t:ordi-hopf}, 
in view of Lemma~\ref{l:submul}. 
An example of a sequence that satisfies~\eqref{e:submult} but whose  Boolean transform fails to be nonnegative is the following:
\[
a_n := \begin{cases}
n & \text{ if $n\leq 4$,} \\
2^n         & \text{ if $n\geq 5$.}
\end{cases}
\]
(The first terms of the Boolean transform are $b_1=1$, $b_2=1$, $b_3=0$, $b_4=-1$.)
Thus, the conditions imposed by Theorem~\ref{t:ordi-hopf} are
strictly stronger than~\eqref{e:submult}.

Condition~\eqref{e:exp} is also implied by Theorem~\ref{t:ordi-hopf}, 
in view of Lemma~\ref{l:increasing} (with $w_n=\frac{1}{n!}$). 

On the other hand, condition~\eqref{e:addcond} is \emph{not} implied by 
Theorem~\ref{t:ordi-hopf}. To see this, let $a_n$ be the sequence of Fibonacci numbers,
defined by $a_0=a_1=1$ and
\[
a_n = a_{n-1}+a_{n-2}
\]
for $n\geq 2$. The Boolean transform is nonnegative; indeed, it is simply given by
\[
b_n = \begin{cases}
1 & \text{ if $n\leq 2$,} \\
0 & \text{ otherwise.}
\end{cases}
\]
However, condition~\eqref{e:addcond} is not satisfied.

The previous example shows that there is no connected Hopf monoid with
dimensions given by the Fibonacci sequence. It also provides another example
for which the answer to question~\eqref{e:ordi-hopf2} is negative.

\appendix

\section{The Boolean transform}\label{s:boolean}

We recall the Boolean transform of a sequence and discuss some
consequences of its nonnegativity. We provide an explicit formula for the
Boolean transform of a Hadamard product in terms of the transforms of the factors.

\subsection{Boolean transform and integer compositions}\label{ss:bool-com}

Let $a_n$, $n\geq 1$, be a sequence of scalars. Its \emph{Boolean transform}
is the sequence $b_n$, $n\geq 1$, defined by
\begin{equation}\label{e:bool-tran}
\sum_{n\geq 1} b_n x^n := 1 - \frac{1}{1+\sum_{n\geq 1} a_n x^n}.
\end{equation}
Equivalently, the sequence $b_n$ can be determined recursively from
\begin{equation}\label{e:bool-tran2}
a_n - \sum_{k=1}^{n-1} a_{n-k}b_k - b_n = 0.
\end{equation}

We also refer to the power series $\sum_{n\geq 1} b_n x^n$ as the Boolean
transform of the power series $\sum_{n\geq 1} a_n x^n$.

\begin{remark}
In the literature on noncommutative probability~\cite{SpeWor:1995}, if $a_n$ is the sequence
of \emph{moments} (of a noncommutative random variable), then $b_n$ is
the sequence of \emph{Boolean cumulants}. The Boolean transform is also
called the \emph{$B$-transform}~\cite{Pop:2009}.
\end{remark}

A \emph{composition} of a nonnegative integer $n$ is a sequence $\alpha=(n_1,\ldots,n_k)$
of positive integers such that
\[
n_1+\cdots+n_k=n.
\]
We write $\alpha\vDash n$. 

Given a sequence $a_n$ and a composition $\alpha\vDash n$ as above, we let
\[
a_\alpha := a_{n_1}\cdots a_{n_k}.
\]
The sequence $a_n$ can be recovered from its Boolean transform $b_n$ by
\begin{equation}\label{e:bool-tran3}
a_n = \sum_{\alpha\vDash n} b_\alpha.
\end{equation}


Given compositions $\sigma=(s_1,\ldots,s_j)\vDash s$ and $\tau=(t_1,\ldots,t_k)\vDash t$, their concatenation
\[
\sigma\cdot\tau := (s_1,\ldots,s_j,t_1,\ldots,t_k)
\]
is a composition of $s+t$.

The set of compositions of $n$ is a Boolean lattice under refinement.
The minimum element is the composition $(n)$ and the maximum is $(1,\ldots,1)$. 
The meet operation and concatenation interact as follows:
\begin{equation}\label{e:meet-conc}
(\alpha\cdot\alpha')\wedge(\beta\cdot\beta') = (\alpha\wedge\beta)\cdot(\alpha'\wedge\beta'),
\end{equation}
where $\alpha,\beta\vDash n$ and $\alpha',\beta'\vDash n'$.

\subsection{Consequences of nonnegativity of the Boolean transform}\label{ss:nonnegative}

We say that a real sequence $a_n$ has nonnegative Boolean transform when all the terms
$b_n$ of its Boolean transform are nonnegative.

\begin{lemma}\label{l:increasing}
Let $a_n$ be a real sequence with nonnegative Boolean transform.
Let $w_n$ be a weakly decreasing sequence such that $w_1\leq 1$.
Then the coefficients of the power series
\[
1-\frac{1+\sum_{n\geq 1} w_na_n x^n}{1+\sum_{n\geq 1} a_n x^n} 
\qand
\frac{1 + \sum_{n\geq 1} a_n x^n}{1 + \sum_{n\geq 1} w_na_n x^n}
\]
are nonnegative.
\end{lemma}
\begin{proof}
Let $C(x):=\sum_{n\geq 1} c_n x^n$ denote the first power series above.
Let $b_n$ be the Boolean transform of $a_n$. In view of~\eqref{e:bool-tran},
\[
C(x)=1-\bigl(1+\sum_{n\geq 1} w_n a_n x^n\bigr)\bigl(1-\sum_{n\geq 1} b_n x^n\bigr).
\]
Hence, for $n\geq 1$,
\[
c_n=-w_na_n + \sum_{k=1}^{n-1} w_{n-k}a_{n-k} b_k + b_n.
\]
Combining with~\eqref{e:bool-tran2} we obtain
\begin{align*}
c_n &= -w_n\bigl(\sum_{k=1}^{n-1} a_{n-k}b_k + b_n\bigr) + \sum_{k=1}^{n-1} w_{n-k}a_{n-k} b_k + b_n\\
&= \sum_{k=1}^{n-1} (w_{n-k}-w_n)a_{n-k} b_k + (1-w_n)b_n.
\end{align*}
The nonnegativity of $b_n$ implies that of $a_n$, by~\eqref{e:bool-tran3}.
Also, $w_{n-k}-w_n\geq 0$ and $1-w_n\geq 0$ by hypothesis. Hence $c_n\geq 0$.

The second power series in the statement is $\frac{1}{1-C(x)}$, so its sequence
of nonconstant coefficients is the inverse Boolean transform of $c_n$. The nonnegativity
of these coefficients follows from that of the $c_n$, by~\eqref{e:bool-tran3}.
\end{proof}

\begin{lemma}\label{l:submul}
Let $a_n$ be a real sequence with nonnegative Boolean transform. Then
\[
a_s a_t \leq a_n
\]
for $n=s+t$ and every $s,t\geq 1$.
\end{lemma}
\begin{proof}
According to~\eqref{e:bool-tran3}, we have
\[
a_s a_t = \bigl(\sum_{\sigma\vDash s} b_\sigma\bigr) \bigl(\sum_{\tau\vDash t} b_\tau\bigr) = \sum_{\substack{\sigma\vDash s\\ \tau\vDash t}} b_{\sigma\cdot \tau}
\leq \sum_{\alpha\vDash n} b_\alpha = a_n.
\]
The inequality holds in view of the nonnegativity of the $b_n$ and the fact that each $\sigma\cdot\tau$ is a distinct composition of $n$.
\end{proof}

\subsection{The Boolean transform and Hadamard products}\label{ss:bool-had}

Let $a_n$ and $b_n$ be two sequences, $n\geq 1$, and let $p_n$ and $q_n$
denote their Boolean transforms. Consider the Hadamard product $a_nb_n$ of
the given sequences, and let $r_n$ denote its Boolean transform. We provide
an explicit formula for $r_n$ in terms of the sequences $p_n$ and $q_n$.

\begin{proposition}\label{p:bool-had}
With the notation as above,
\begin{equation}\label{e:bool-had}
r_n = \sum_{\substack{\alpha,\beta\vDash n\\ \alpha\wedge\beta=(n)}} p_\alpha q_\beta.
\end{equation}
\end{proposition}
\begin{proof}
Define, for each $\gamma\vDash n$, a scalar 
\[
\Tilde{r}_\gamma:= \sum_{\substack{\alpha,\beta\vDash n\\ \alpha\wedge\beta=\gamma}} p_\alpha q_\beta.
\]

Fix two compositions $\gamma\vDash n$ and $\gamma'\vDash n'$. Let 
$n'':=n+n'$ and $\gamma'':=\gamma\cdot\gamma'\vDash n''$.
Let $\alpha,\beta\vDash n$ and $\alpha',\beta'\vDash n'$ be compositions such that
\[
\gamma= \alpha\wedge\beta
\qand
\gamma'= \alpha'\wedge\beta'.
\]
Let $\alpha'':=\alpha\cdot\alpha'$ and $\beta'':=\beta\cdot\beta'$. 
Then, by~\eqref{e:meet-conc},
\[
\alpha''\wedge\beta'' = (\alpha\cdot\alpha')\wedge(\beta\cdot\beta') =
(\alpha\wedge\beta)\cdot(\alpha'\wedge\beta')=
 \gamma\cdot\gamma' = \gamma''.
\]
Conversely, any pair of compositions $\alpha'',\beta''\vDash n''$ such that
$\alpha''\wedge\beta''= \gamma''$ arises as above from unique $\alpha,\alpha',\beta$ and $\beta'$. It follows that
\[
\Tilde{r}_\gamma \Tilde{r}_{\gamma'} = \sum_{\substack{\alpha,\beta\vDash n,\, 
\alpha',\beta'\vDash n'\\
\alpha\wedge\beta=\gamma,\, \alpha'\wedge\beta'=\gamma'}} p_\alpha q_\beta p_{\alpha'} q_{\beta'}
=  \sum_{\substack{\alpha'',\beta''\vDash n''\\\alpha''\wedge\beta''=\gamma''}} p_{\alpha''} q_{\beta''} = \Tilde{r}_{\gamma''}.
\]
This implies that, for $\gamma=(n_1,\ldots,n_k)$,
\[
\Tilde{r}_\gamma = \Tilde{r}_{(n_1)}\cdots \Tilde{r}_{(n_k)}.
\]

On the other hand, from the definition of $\Tilde{r}$ and~\eqref{e:bool-tran3}
we have that
\[
\sum_{\gamma\vDash n} \Tilde{r}_\gamma = \sum_{\alpha,\beta\vDash n} p_\alpha q_\beta = a_n b_n.
\]

The previous two equalities imply that the sequence $a_n b_n$ is the
inverse Boolean transform of the sequence $\Tilde{r}_{(n)}$, in view of~\eqref{e:bool-tran3}. Thus, $\Tilde{r}_{(n)}$ is the Boolean transform of $a_n b_n$ and the result follows.
\end{proof}

The first values of $r_n$ are as follows:
\begin{gather*}
r_1 = p_1q_1,\\
r_2 = p_2q_2+ p_2 q_1^2 + p_1^2 q_2,\\
r_3 = p_3 q_3 + 2p_3 q_2q_1 + 2 p_2p_1q_3 + 2p_2p_1q_2q_1 + p_1^3 q_3 + p_3 q_1^3.
\end{gather*}

\begin{corollary}\label{c:submonoid}
The set of real sequences with nonnegative Boolean transform
is closed under Hadamard products.
\end{corollary}
\begin{proof}
This follows at once from~\eqref{e:bool-had}.
\end{proof}


\bibliographystyle{plain}  
\bibliography{hadamard}

\end{document}